\documentclass[10pt]{amsart}

\usepackage{amsmath,amssymb,amsthm,mathtools}
\usepackage[colorlinks=true,linkcolor=blue,citecolor=blue,urlcolor=blue]{hyperref}
\usepackage{microtype}

\newtheorem{theorem}{Theorem}[section]
\newtheorem{proposition}[theorem]{Proposition}
\newtheorem{lemma}[theorem]{Lemma}
\newtheorem{corollary}[theorem]{Corollary}
\newtheorem{claim}[theorem]{Claim}
\theoremstyle{remark}
\newtheorem{remark}[theorem]{Remark}

\newcommand{\R}{\mathbb{R}}
\newcommand{\Q}{\mathbb{Q}}
\newcommand{\Z}{\mathbb{Z}}
\newcommand{\T}{\mathbb{T}}
\newcommand{\PP}{\mathbb{P}}
\newcommand{\EE}{\mathbb{E}}
\newcommand{\rank}{\operatorname{rank}}
\newcommand{\codim}{\operatorname{codim}}
\newcommand{\Span}{\operatorname{span}}
\newcommand{\dist}{\operatorname{dist}}
\newcommand{\1}{\mathbf{1}}

\title{Exponential Rank Bounds for Random Matrices}
\author{Achintya Raya Polavarapu}
\email{apolavarapu6@gatech.edu}
\address{School of Mathematics, Georgia Institute of Technology, Atlanta GA 30332.}
\date{}

\begin{document}

\begin{abstract}
Fix \(b\in(0,1)\), let \(1\leq k\leq n\), and let \(A=(A_{ij})\) be an
\(n\times n\) random matrix with independent real entries satisfying
\[
   \sup_{x\in\R}\PP\{A_{ij}=x\}\leq b<1
   \qquad (1\leq i,j\leq n).
\]
We show that there exists \(c>0\) such that
\[
   \PP\{\rank A\leq n-k\}\leq \exp(-cnk),\qquad 1\leq k\leq n.
\]
\end{abstract}

\maketitle
\vspace{-0.9\baselineskip}

\section{Introduction}

\noindent
The singularity problem for random matrices begins with Koml\'os
\cite{Komlos1967}, who showed that an \(n\times n\) Bernoulli matrix is
nonsingular with probability tending to \(1\).  Kahn--Koml\'os--Szemer\'edi
\cite{KKS1995} later proved the first exponential bound on the singularity
probability.  After successive refinements by Tao--Vu \cite{TV06,TV07} and
Bourgain--Vu--Wood \cite{BVW2010}, Tikhomirov \cite{Tikhomirov2020}
established the sharp Bernoulli asymptotic \((1/2+o(1))^n\).  For general
discrete iid laws with finite support, Jain--Sah--Sawhney \cite{JSS2020}
showed that the singularity probability is asymptotically governed by
repeated rows and repeated columns.  On the least-singular-value side,
Rudelson--Vershynin \cite{RV08,RV09} proved optimal bounds in the
rectangular and square subgaussian settings, and Tao--Vu \cite{TV10a}
proved universality for the rescaled least singular value.

\medskip\noindent
Substantial progress has also been made in the inhomogeneous setting, where
the entries are independent but not necessarily identically distributed.
Livshyts--Tikhomirov--Vershynin \cite{LTV2021} proved an exponential
singularity estimate under second-moment and L\'evy-type assumptions, via
least-singular-value bounds.  Under the weaker hypothesis that no entry
places mass exceeding \(1-\varepsilon\) on any point,
Hunter--Kwan--Sauermann \cite[Theorem~1.2]{HKS2025perm} proved exponential
anticoncentration for the permanent, and observed that their argument
applies equally to the determinant; in particular, the singularity
probability is exponentially small under this uniform atom bound alone,
with no moment assumptions.

\medskip\noindent
In this paper we study the rank-deficiency probabilities
\[
   \PP\{\rank A\leq n-k\},
   \qquad 1\leq k\leq n.
\]
The case \(k=1\) recovers the singularity problem; for \(k>1\) one asks for
an exact rank deficiency of order \(k\), so that all columns of
\(A\) lie in a common \((n-k)\)-dimensional subspace.

\medskip\noindent
Considerably less is known once \(k\) is allowed to grow.  In the iid
subgaussian setting, Rudelson \cite{Rudelson2023} proved an exponential
large-deviation bound for the rank in the range \(k\leq c\sqrt n\).  For
fixed \(k\) in the Bernoulli model, Jain--Sah--Sawhney \cite{JSS22}
identified the sharp exponential rate, and
Hunter--Kwan--Sauermann--Sawhney \cite{HKSS2025} more recently extended
this to the full range \(1\leq k\leq n\) in the Bernoulli case.

\medskip\noindent
Our main result shows that the same uniform atom bound also yields
exponential rank-deficiency bounds for the full range of \(k\) in the
inhomogeneous setting.

\begin{theorem}[Independent atom-bounded entries]\label{thm:main-atom}
Fix \(b\in(0,1)\).  For each \(n\ge1\), let \(A=(A_{ij})\) be an
\(n\times n\) random matrix with independent real entries such that
\[
   \sup_{x\in\R}\PP\{A_{ij}=x\}\leq b<1
   \qquad (1\leq i,j\leq n).
\]
Then there exists \(c_{1.1}=c_{1.1}(b)>0\) such that for all \(n\ge1\) and
all \(1\leq k\leq n\),
\begin{equation}\label{eq:atom-large-corank}
   \PP\{\rank A\leq n-k\}\leq \exp(-c_{1.1}nk).
\end{equation}
\end{theorem}

\medskip\noindent
Since any square submatrix \(B\) of \(A\) satisfies \(\rank B\leq \rank A\),
applying Theorem~\ref{thm:main-atom} to a fixed \(d\times d\) square
submatrix yields the following rectangular version.

\begin{corollary}[Rectangular matrices]\label{cor:rectangular}
Fix \(b\in(0,1)\).  Let \(A\) be an \(m\times n\) random matrix with
independent real entries such that
\[
   \sup_{x\in\R}\PP\{A_{ij}=x\}\leq b<1
   \qquad (1\leq i\leq m,\ 1\leq j\leq n),
\]
and set \(d=\min(m,n)\).  Then for all \(1\leq k\leq d\),
\begin{equation}\label{eq:rectangular-corank}
   \PP\{\rank A\leq d-k\}\leq \exp(-c_{1.1}dk).
\end{equation}
\end{corollary}

\medskip\noindent
The rest of the paper is organized as follows.
Section~\ref{sec:prelim-overview} collects notation, gives an overview of
the proof, and records several preliminary lemmas.  Section~\ref{sec:bbrac}
establishes the Bernoulli relative anticoncentration estimate with
coordinatewise biases.  Section~\ref{sec:comparison-thin} develops the
comparison argument for the inhomogeneous column laws together with the
elementary-symmetric thin/thick decomposition.
Section~\ref{sec:proof-main} completes the proof of
Theorem~\ref{thm:main-atom}.

\subsection*{Acknowledgements}
The author thanks Galyna Livshyts for her guidance and encouragement, and
Zach Hunter and Matthew Kwan for helpful comments on an earlier draft.

\medskip\noindent
The author used OpenAI's Codex to assist with writing and revising the paper
and discussing proof details. The author takes full
responsibility for the content of the paper.

\section{Preliminaries and proof overview}\label{sec:prelim-overview}

\subsection{Notation}
All vectors are column vectors.  We write \(\T=\R/\Z\), and we use \(\mu\) for the normalized
Lebesgue measure on \(\T\) and, more generally, on \(\T^k\) for every \(k\ge1\).  For \(x\in\T\),
\(\|x\|_{\T}\) denotes the distance from \(x\) to \(0\) in \(\T\).  If
\(x_1,\ldots,x_r\in\R^n\), then \(\Span(x_1,\ldots,x_r)\) denotes their linear
span.  For a linear subspace \(V\subseteq\R^n\) and a vector \(a\in\R^n\), the
translate \(a+V\) is the corresponding affine subspace.  Unless explicitly
stated otherwise, every subspace in the paper is a linear subspace of
\(\R^n\).

\medskip\noindent
For \(0\leq d\leq n\), write
\[
   \mathcal S(n,d)=\{V\subseteq\R^n:\dim V=d\}
\]
for the family of \(d\)-dimensional subspaces of \(\R^n\).  When a measurable
structure on \(\mathcal S(n,d)\) is needed, we identify each subspace
\(V\in\mathcal S(n,d)\) with its orthogonal projection \(P_V\) and use the
induced Borel structure.  We call \(V\subseteq\R^n\) rational if it is
spanned by vectors in \(\Q^n\), or equivalently if
\(V=\ker L\) for some matrix \(L\) with rational entries.  Finally, for a
real random variable
\(\eta\), write
\[
   Q(\eta)=\sup_{u\in\R}\PP\{\eta=u\}.
\]
We say that a family of random variables satisfies a uniform atom bound if
there exists \(b<1\) such that \(Q(\eta)\leq b\) uniformly over the family.

\medskip\noindent
\subsection{Proof Overview}

\medskip\noindent
If \(\rank A\leq n-k\), then all columns of \(A\) lie in a common
\((n-k)\)-dimensional subspace \(V\subseteq\R^n\).  Thus the problem becomes a
question about the probability that many independent random columns all land
in one low-codimension subspace.

\medskip\noindent
The argument starts from the Fourier/doubling and thin/thick method of
Hunter--Kwan--Sauermann--Sawhney \cite{HKSS2025} and adapts it to fully
independent inhomogeneous matrices.  Under the uniform atom bound, each entry
admits a Bernoulli decomposition with a random shift and a random scale.  After
conditioning on this extra randomness, one is reduced to columns built from
independent Bernoulli coordinates whose parameters stay uniformly away from
\(0\) and \(1\).

\medskip\noindent
Proposition~\ref{prop:bbrac} is the main estimate for a single column.  It compares the probability
that such a Bernoulli vector lies in a codimension-\(k\) affine slice with the
corresponding probability for a symmetric three-point comparison law, and this
already yields an exponential loss in the codimension \(k\).

\medskip\noindent
The rest of the proof turns this one-column estimate into a statement about all
columns at once by splitting the possible subspaces according to how
likely they are to capture a column from the comparison distribution.  If that
probability is very small, then independence already makes it
unlikely that every column lands in the same subspace.  The remaining
subspaces are those that the comparison distribution hits with relatively large
probability.  One must then show that there are still few enough such
subspaces to sum over them.  In the iid
Bernoulli setting of \cite{HKSS2025}, this step can be organized using a
single hit probability attached to the subspace.  In the present
inhomogeneous setting, one has to track the corresponding columnwise
probabilities simultaneously.

\medskip\noindent
This argument already yields the theorem for all sufficiently large \(k\) from
the uniform atom bound alone.  To cover the finitely many smaller values of
\(k\), we combine it with the exponential singularity estimate that is
already available under the same hypothesis.

\subsection{Preliminary Results}
\medskip\noindent
We begin by recording the measurability facts that let us treat
random spans and subspace-valued events without further comment in the main
argument.

\begin{lemma}[Measurability in the subspace variable]\label{lem:grassmann-meas}
Fix \(0\leq d\leq n\), and let \(U\) be an \(\R^n\)-valued random vector.  Then
the map
\[
   V\mapsto \PP\{U\in V\},
   \qquad
   V\in\mathcal S(n,d),
\]
is Borel.  More generally, if \(U_1,\ldots,U_r\) are independent random
vectors in \(\R^n\), then
\[
   V\mapsto \PP\{U_1,\ldots,U_r\in V\}
\]
is Borel on \(\mathcal S(n,d)\).
\end{lemma}

\begin{proof}
For \(x\in\R^n\),
\[
   \dist(x,V)=\|(I-P_V)x\|_2.
\]
Hence \((V,x)\mapsto \dist(x,V)\) is continuous from
\(\mathcal S(n,d)\times\R^n\) to \(\R\).  The indicator
\[
   \1_{\{x\in V\}}=\1_{\{\dist(x,V)=0\}}
\]
is therefore Borel on \(\mathcal S(n,d)\times\R^n\).  Integrating this bounded
measurable function against the law of \(U\) yields the first statement.  The
second follows in the same way by replacing \(\1_{\{x\in V\}}\) with the
product indicator \(\prod_{j=1}^r \1_{\{x_j\in V\}}\).
\end{proof}

\medskip\noindent
We will also need the span map to be measurable.

\begin{lemma}[Measurability of random spans]\label{lem:span-map}
Fix \(1\leq d\leq n\).  Let
\[
   \Omega_d=\{(x_1,\ldots,x_d)\in(\R^n)^d:\dim\Span(x_1,\ldots,x_d)=d\}.
\]
Then the map
\[
   (x_1,\ldots,x_d)\mapsto \Span(x_1,\ldots,x_d),
   \qquad
   (x_1,\ldots,x_d)\in\Omega_d,
\]
is Borel from \(\Omega_d\) to \(\mathcal S(n,d)\).
\end{lemma}

\begin{proof}
Identify \((x_1,\ldots,x_d)\in\Omega_d\) with the matrix
\[
   A=[x_1\ \cdots\ x_d]\in\R^{n\times d}.
\]
On \(\Omega_d\), the matrix \(A\) has full column rank, and the orthogonal
projection onto \(\Span(x_1,\ldots,x_d)\) is
\[
   P_A=A(A^\top A)^{-1}A^\top.
\]
This depends continuously on \(A\), so the span map is continuous, hence
Borel.
\end{proof}

\medskip\noindent
We next record the elementary codimension bound for product measures that we
use throughout the paper.  It goes back to Odlyzko \cite{Odlyzko1988}.

\begin{lemma}[Odlyzko bound]\label{lem:odlyzko}
Let \(U=(U_1,\ldots,U_n)\) have independent real coordinates.  Suppose that
for some \(\lambda\in(0,1)\),
\[
   \sup_{x\in\R}\PP\{U_i=x\}\leq \lambda
   \qquad\text{for every }i.
\]
Then for every affine subspace \(W\subseteq\R^n\) of codimension \(r\),
\[
   \PP\{U\in W\}\leq \lambda^r.
\]
\end{lemma}

\begin{proof}
Write \(W=\{x\in\R^n:Ax=b\}\), where \(A\) is an \(r\times n\) real matrix of
rank \(r\).  After permuting coordinates we may assume that the first \(r\)
columns of \(A\) are linearly independent.  Conditional on
\((U_{r+1},\ldots,U_n)\), membership in \(W\) forces \((U_1,\ldots,U_r)\) to
equal a single prescribed vector.  By independence, this conditional
probability is at most \(\lambda^r\).  Averaging proves the claim.
\end{proof}

\medskip\noindent
In the thick-subspace argument, a factor \(\bar\rho(V)^m\) appears on both
sides of the comparison.  The next lemma lets us cancel that same weight.

\begin{lemma}[Weighted cancellation]\label{lem:weighted-cancel}
Let \(\Omega\) be a measurable space, let \(\nu\) and \(\sigma\) be finite
measures on \(\Omega\), and let \(w:\Omega\to(0,\infty)\) be bounded and
Borel.  Suppose
that for some \(C>0\),
\[
   \int_A w\,d\nu\leq C\int_A w\,d\sigma
\]
for every Borel set \(A\subseteq\Omega\).  Then
\[
   \nu(A)\leq C\sigma(A)
\]
for every Borel set \(A\subseteq\Omega\).
\end{lemma}

\begin{proof}
Define finite measures
\[
   \widetilde\nu(A)=\int_A w\,d\nu,
   \qquad
   \widetilde\sigma(A)=\int_A w\,d\sigma.
\]
By assumption, \(\widetilde\nu(A)\leq C\widetilde\sigma(A)\) for every Borel
set \(A\), so \(\widetilde\nu\ll\widetilde\sigma\).  If \(\sigma(A)=0\), then
\(\widetilde\sigma(A)=0\), hence \(\widetilde\nu(A)=0\).  Since \(w>0\) on
\(A\), this forces \(\nu(A)=0\): otherwise
\[
   A=\bigcup_{m=1}^\infty \left(A\cap\left\{w\ge\frac1m\right\}\right)
\]
would imply that \(\nu(A\cap\{w\ge 1/m\})>0\) for some \(m\), and therefore
\[
   \widetilde\nu(A)\ge \frac1m \nu\!\left(A\cap\left\{w\ge\frac1m\right\}\right)>0,
\]
a contradiction.  Thus \(\nu\ll\sigma\).

\medskip\noindent
Let \(f=d\nu/d\sigma\).  Then for every Borel set \(A\),
\[
   \int_A wf\,d\sigma
   =
   \int_A w\,d\nu
   \leq
   C\int_A w\,d\sigma.
\]
Hence \(wf\leq Cw\) \(\sigma\)-almost everywhere, so \(f\leq C\)
\(\sigma\)-almost everywhere because \(w>0\).  Therefore
\[
   \nu(A)=\int_A f\,d\sigma\leq C\sigma(A)
\]
for every Borel set \(A\subseteq\Omega\).
\end{proof}

\medskip\noindent
The Fourier argument produces level sets in \(\T^k\).  The following doubling
lemma from \cite{HKSS2025} is what turns information about those sets into an
exponential gain in codimension.

\begin{lemma}[HKSS doubling lemma]\label{lem:hkss-doubling}
Let \(A_1,\ldots,A_k\subseteq\T\) be closed sets with \(\mu(A_i)\leq 1/2\) for
all \(i\).  If \(S\subseteq A_1\times\cdots\times A_k\subseteq\T^k\) is
closed, then
\[
   \mu(S+S)\ge 2^k\mu(S).
\]
\end{lemma}

\medskip\noindent
To apply the Bernoulli argument to a matrix with only an atom bound, we first
rewrite each entry, after introducing extra randomness, in terms of a
Bernoulli variable whose bias stays away from \(0\) and \(1\).  The next
lemma provides this one-dimensional reduction.  It is a specialization of the
Bernoulli decomposition theorem of Aizenman--Germinet--Klein--Warzel
\cite{AGKW2007}; we include a direct proof because it is short.

\begin{lemma}[Bernoulli decomposition from an atom bound]\label{lem:atom-decomp}
Let \(\xi\) be a real random variable with \(Q(\xi)\leq b<1\), and put
\[
   \rho=\frac{1-b}{2}.
\]
Then there exist \(p\in[\rho,1-\rho]\), measurable functions
\(f,\delta:(0,1)\to\R\) with \(\delta(t)\neq0\) for all \(t\in(0,1)\), and
independent random variables \(T\sim{\rm Unif}(0,1)\) and
\(\varepsilon\sim{\rm Bernoulli}(p)\) such that
\[
   \xi\stackrel d=f(T)+\delta(T)\varepsilon.
\]
\end{lemma}

\medskip\noindent
In what follows, we use only the
measurability of \(f\) and \(\delta\), the uniform bound
\(p\in[\rho,1-\rho]\), and the nonvanishing condition \(\delta(t)\neq0\),
which ensures that the resulting diagonal rescalings are invertible after
conditioning.

\begin{proof}
Let \(\mu\) be the law of \(\xi\), and let
\[
   F(x)=\mu((-\infty,x]).
\]
Since \(\mu(\{x\})\leq Q(\xi)\leq b\) for every \(x\in\R\), each jump of
\(F\) has size at most \(b\).  Let
\[
   x_0=\inf\{x\in\R:F(x)\ge \rho\},
\]
and put \(E=(-\infty,x_0]\).  Then
\[
   \rho\leq \mu(E)=F(x_0)\leq \rho+b=1-\rho.
\]
Set
\[
   p=\mu(E)\in[\rho,1-\rho].
\]
Define probability measures
\[
   \mu_1(A)=\frac{\mu(A\cap E)}{p},
   \qquad
   \mu_0(A)=\frac{\mu(A\cap E^c)}{1-p}.
\]
Choose measurable functions \(f_0,f_1:(0,1)\to\R\) with pushforwards
\(\mu_0,\mu_1\), respectively.  Since \(\mu_0\) is supported on \(E^c\) and
\(\mu_1\) is supported on \(E\), after modifying on null sets we may assume
\[
   f_0(t)\in E^c,
   \qquad
   f_1(t)\in E
\]
for every \(t\in(0,1)\).  In particular, \(f_1(t)\neq f_0(t)\) for all
\(t\in(0,1)\).  Setting
\[
   f=f_0,
   \qquad
   \delta=f_1-f_0
\]
and letting \(T\sim{\rm Unif}(0,1)\) and
\(\varepsilon\sim{\rm Bernoulli}(p)\) be independent, we have that \(f(T)\)
has law \(\mu_0\) and \(f(T)+\delta(T)=f_1(T)\) has law \(\mu_1\).  Therefore
\[
   f(T)+\delta(T)\varepsilon
\]
has law
\[
   (1-p)\mu_0+p\mu_1=\mu,
\]
which is the law of \(\xi\).
\end{proof}

\section{Bernoulli relative anticoncentration}\label{sec:bbrac}

\noindent Fix \(0<\rho\leq 1/2\).  Let \(B=(B_1,\ldots,B_n)\) have independent
coordinates with
\[
   B_j\sim {\rm Bernoulli}(q_j),
   \qquad
   q_j\in[\rho,1-\rho]
   \quad (1\leq j\leq n).
\]
For a parameter \(\alpha\in(0,1/4)\), let \(Z=(Z_1,\ldots,Z_n)\) have iid
coordinates with the lazy symmetric three-point law
\begin{equation}\label{eq:lazy-def-noniid}
   \PP\{Z_i=0\}=1-2\alpha,
   \qquad
   \PP\{Z_i=1\}=\PP\{Z_i=-1\}=\alpha.
\end{equation}

\medskip\noindent
After the Bernoulli reduction, the basic question is how likely such a vector
is to lie in a fixed affine slice.  The next proposition gives the bound
we need: the probability pays an exponential price in the codimension.

\begin{proposition}[Bernoulli relative anticoncentration]\label{prop:bbrac}
For every \(0<\rho\leq 1/2\) there are constants
\(\alpha_{3.1}=\alpha_{3.1}(\rho)>0\),
\(\gamma_{3.1}=\gamma_{3.1}(\rho)<1\), and
\(k_{3.1}=k_{3.1}(\rho)\) such that the following holds.  Let
\(B=(B_1,\ldots,B_n)\) have independent coordinates with
\[
   B_j\sim {\rm Bernoulli}(q_j),
   \qquad
   q_j\in[\rho,1-\rho]
   \quad (1\leq j\leq n),
\]
and let \(Z\) be defined by \eqref{eq:lazy-def-noniid} with
\(\alpha=\alpha_{3.1}\).  If \(V\subseteq\R^n\) is a linear subspace of
codimension \(k\ge k_{3.1}\), then
\begin{equation}\label{eq:bbrac-noniid}
   \sup_{a\in\R^n}\PP\{B\in a+V\}
   \leq
   \gamma_{3.1}^k \PP\{Z\in V\}.
\end{equation}
\end{proposition}

\begin{proof}
\medskip\noindent
The argument below follows the same general template as the proof of
Proposition~2.1 in \cite{HKSS2025}, with the unbiased Bernoulli factors there
replaced by the biased one-dimensional Fourier factors
\eqref{eq:phi-h-def-noniid}.  We first treat
the rational case.  Assume that \(V\subseteq\R^n\) is a
rational linear subspace of codimension \(k\).  Choose an integer matrix
\(L\in\Z^{k\times n}\) with \(\ker L=V\).  Permuting coordinates, applying
rational row operations, and then multiplying by a common positive
denominator, we may assume that the first \(k\) columns of \(L\) are
\[
   N_0 e_1,\ldots,N_0 e_k
\]
for some positive integer \(N_0\).  Let \(w_1,\ldots,w_n\in\Z^k\) denote the
columns of \(L\).

\medskip\noindent
Define
\begin{equation}\label{eq:phi-h-def-noniid}
   \phi_p(u)=1-p+pe^{2\pi i u},
   \qquad
   h_\alpha(u)=1-2\alpha+2\alpha\cos(2\pi u).
\end{equation}
Since \(\alpha<1/4\), \(h_\alpha(u)\ge1-4\alpha>0\) for all \(u\in\T\).

\medskip\noindent
Fourier inversion gives, for every \(s\in\Z^k\),
\[
   \PP\{LB=s\}
   =
   \int_{\T^k} e^{-2\pi i\theta\cdot s}
      \prod_{j=1}^n \phi_{q_j}(\theta\cdot w_j)\,d\theta,
\]
so
\begin{equation}\label{eq:fourier-upper-B-noniid}
   \sup_{a\in\R^n}\PP\{B\in a+V\}
   \leq
   \int_{\T^k}\prod_{j=1}^n |\phi_{q_j}(\theta\cdot w_j)|\,d\theta.
\end{equation}
Indeed, if \((a+V)\cap\{0,1\}^n=\varnothing\), then the probability is zero.
Otherwise \(La=Lb\in\Z^k\) for every \(b\in(a+V)\cap\{0,1\}^n\), so
\(\PP\{B\in a+V\}=\PP\{LB=La\}\).  Similarly,
\begin{equation}\label{eq:fourier-Z-noniid}
   \PP\{Z\in V\}
   =
   \int_{\T^k}\prod_{j=1}^n h_\alpha(\theta\cdot w_j)\,d\theta.
\end{equation}

\medskip\noindent
Let
\[
   A_\rho=2\rho(1-\rho).
\]
For \(p\in[\rho,1-\rho]\), put
\[
   A_p=2p(1-p).
\]
Then \(A_p\ge A_\rho>0\).  Fix once and for all
\begin{equation}\label{eq:alpha-choice-noniid}
   0<\alpha\leq \min\left\{\frac1{128},\frac{A_\rho}{64}\right\}.
\end{equation}
Put
\[
   D(u)=1-\cos(2\pi u).
\]
Then \(D(u)\in[0,2]\), and
\[
   |\phi_p(u)|^2=1-A_pD(u),
   \qquad
   h_\alpha(u)=1-2\alpha D(u).
\]

\medskip\noindent
The proof now reduces to a pointwise comparison between the Fourier factors of
the Bernoulli law and those of the lazy comparison law.

\begin{lemma}[One-dimensional comparison]\label{lem:onedim-comparison-noniid}
With \(\alpha\) as in \eqref{eq:alpha-choice-noniid}, for all
\(p,p'\in[\rho,1-\rho]\) and all \(u,v\in\T\),
\begin{align}
   |\phi_p(u)| &\leq h_\alpha(u)^2, \label{eq:pointwise-one-noniid}\\
   |\phi_p(u)|\,|\phi_{p'}(v)| &\leq h_\alpha(u+v)^2.
   \label{eq:pointwise-two-noniid}
\end{align}
\end{lemma}

\begin{proof}
Since \(A_pD(u)\in[0,1]\),
\[
   |\phi_p(u)|=\sqrt{1-A_pD(u)}
   \leq 1-\frac{A_pD(u)}2
   \leq 1-\frac{A_\rho D(u)}2.
\]
On the other hand,
\[
   h_\alpha(u)^2=(1-2\alpha D(u))^2
   \ge 1-4\alpha D(u)
   \ge 1-\frac{A_\rho D(u)}2,
\]
because \(4\alpha\leq A_\rho/2\).  This proves
\eqref{eq:pointwise-one-noniid}.

\medskip\noindent
For \eqref{eq:pointwise-two-noniid}, set \(x=A_pD(u)\) and
\(y=A_{p'}D(v)\).  Then \(x,y\in[0,1]\).  We first note that
\begin{equation}\label{eq:sqrt-xy-noniid}
   \sqrt{(1-x)(1-y)}\leq 1-\frac{x+y}{4}.
\end{equation}
Indeed, the right-hand side is at least \(1/2\), and after squaring the
desired inequality follows from
\[
   (1-x)(1-y)
   \leq
   1-(x+y)+\frac{(x+y)^2}{4}
   \leq
   \left(1-\frac{x+y}{4}\right)^2.
\]
Hence
\begin{equation}\label{eq:phi-prod-upper-noniid}
   |\phi_p(u)|\,|\phi_{p'}(v)|
   \leq
   1-\frac{A_\rho}{4}(D(u)+D(v)).
\end{equation}
Also
\begin{equation}\label{eq:D-subadd-noniid}
   D(u+v)\leq 2D(u)+2D(v).
\end{equation}
Therefore
\[
   h_\alpha(u+v)^2
   \ge
   1-4\alpha D(u+v)
   \ge
   1-8\alpha(D(u)+D(v)).
\]
Since \(8\alpha\leq A_\rho/4\), this lower bound is at least the right-hand
side of \eqref{eq:phi-prod-upper-noniid}.  This proves
\eqref{eq:pointwise-two-noniid}.
\end{proof}

\medskip\noindent
Now set
\[
   F(\theta)=\prod_{j=1}^n |\phi_{q_j}(\theta\cdot w_j)|,
   \qquad
   G(\theta)=\prod_{j=1}^n h_\alpha(\theta\cdot w_j).
\]
By \eqref{eq:pointwise-one-noniid},
\begin{equation}\label{eq:F-G-square-noniid}
   F(\theta)\leq G(\theta)^2
   \qquad\text{for all }\theta\in\T^k.
\end{equation}
We prove that, for suitable \(\gamma<1\) and all large enough \(k\),
\begin{equation}\label{eq:integral-comparison-noniid}
   \int_{\T^k}F(\theta)\,d\theta
   \leq
   \gamma^k\int_{\T^k}G(\theta)\,d\theta.
\end{equation}

\medskip\noindent
Let
\[
   r_\rho=\sqrt{1-A_\rho}<1.
\]
Choose \(\delta_0\in(0,1/2)\) so small that
\[
   H(\delta_0)<\log(3/2),
\]
where \(H(t)=-t\log t-(1-t)\log(1-t)\) is the binary entropy function.  Then
choose
\[
   0<\beta<\delta_0(-\log r_\rho).
\]
Put \(\tau=e^{-\beta k}\).  We split
\[
   \int F
   =
   \int \min(F,\tau)+\int(F-\tau)_+.
\]
For the low part, \eqref{eq:F-G-square-noniid} gives
\begin{equation}\label{eq:low-part-noniid}
   \int_{\T^k}\min(F,\tau)
   \leq
   \tau^{1/2}\int_{\T^k} G.
\end{equation}

\medskip\noindent
For the high part, for \(\eta\in[\tau,1]\), define the closed level set
\[
   S_\eta=\{\theta\in\T^k:F(\theta)\ge \eta\}.
\]
Since the first \(k\) columns of \(L\) are \(N_0e_1,\ldots,N_0e_k\), for
\(\theta\in S_\eta\),
\[
   \prod_{i=1}^k |\phi_{q_i}(N_0\theta_i)|
   \ge
   F(\theta)
   \ge
   e^{-\beta k}.
\]
If \(\|N_0x\|_{\T}>1/4\), then \(D(N_0x)\ge1\), and therefore
\[
   |\phi_{q_i}(N_0x)|\leq \sqrt{1-A_{q_i}}\leq r_\rho.
\]
Thus each \(\theta\in S_\eta\) has at most
\[
   M=\left\lfloor\frac{\beta k}{-\log r_\rho}\right\rfloor
\]
coordinates \(i\) with \(\|N_0\theta_i\|_{\T}>1/4\).  By the choice of
\(\beta\), \(M\leq \delta_0k\).  For all sufficiently large \(k\),
\[
   \sum_{j=0}^{M}\binom{k}{j}
   \leq
   \exp\!\bigl(H(M/k)k\bigr)
   \leq
   \exp\!\bigl(H(\delta_0)k\bigr)
   \leq
   (3/2)^k,
\]
by the standard entropy bound on the lower tail of the binomial coefficients.

\medskip\noindent
For \(I\subseteq[k]\) with \(|I|\leq M\), let \(B_I\) be the set of
\(\theta\in\T^k\) such that
\[
   \|N_0\theta_i\|_{\T}>1/4 \quad (i\in I),
   \qquad
   \|N_0\theta_i\|_{\T}\leq 1/4 \quad (i\notin I).
\]
Then \(S_\eta\subseteq \bigcup_{|I|\leq M} B_I\).  Hence for every
\(\eta\in[\tau,1]\) there is some \(I\) with
\begin{equation}\label{eq:pigeon-noniid}
   \mu(S_\eta\cap B_I)\ge (2/3)^k \mu(S_\eta).
\end{equation}
Let
\[
   G_0=\{x\in\T:\|N_0x\|_{\T}\leq 1/4\},
   \qquad
   H_0=\{x\in\T:\|N_0x\|_{\T}\ge 1/4\}.
\]
Both are closed and have measure \(1/2\): multiplication by \(N_0\) is
measure-preserving on \(\T\), and the set \(\{x\in\T:\|x\|_{\T}\leq 1/4\}\) has
measure \(1/2\).  For the chosen \(I\), define the
closed box
\[
   C_I=\prod_{i=1}^k A_i,
   \qquad
   A_i=
   \begin{cases}
      H_0, & i\in I,\\
      G_0, & i\notin I.
   \end{cases}
\]
Then \(B_I\subseteq C_I\), so \eqref{eq:pigeon-noniid} implies
\[
   \mu(S_\eta\cap C_I)\ge (2/3)^k \mu(S_\eta).
\]
Applying Lemma~\ref{lem:hkss-doubling} to the closed set \(S_\eta\cap C_I\),
\begin{equation}\label{eq:doubling-level-noniid}
   \mu(S_\eta+S_\eta)\ge (4/3)^k \mu(S_\eta).
\end{equation}

\medskip\noindent
If \(\theta,\theta'\in S_\eta\), then
\eqref{eq:pointwise-two-noniid} gives
\[
   h_\alpha((\theta+\theta')\cdot w_j)^2
   \ge
   |\phi_{q_j}(\theta\cdot w_j)|\,|\phi_{q_j}(\theta'\cdot w_j)|
\]
for every \(j\).  Multiplying over \(j\) yields
\[
   G(\theta+\theta')\ge F(\theta)^{1/2}F(\theta')^{1/2}\ge \eta.
\]
Therefore
\[
   S_\eta+S_\eta\subseteq \{\theta\in\T^k:G(\theta)\ge \eta\}.
\]
Combining this with \eqref{eq:doubling-level-noniid},
\[
   \mu(S_\eta)\leq (3/4)^k \mu\{G\ge \eta\}.
\]
Using the layer-cake representation,
\begin{align}
   \int_{\T^k}(F-\tau)_+\,d\theta
   &=\int_\tau^1 \mu\{\theta\in\T^k:F(\theta)\ge \eta\}\,d\eta \notag\\
   &=\int_\tau^1 \mu(S_\eta)\,d\eta \notag\\
   &\leq (3/4)^k \int_\tau^1 \mu\{G\ge\eta\}\,d\eta
   \leq (3/4)^k \int_{\T^k}G(\theta)\,d\theta.
   \label{eq:high-part-noniid}
\end{align}
Together with \eqref{eq:low-part-noniid},
\[
   \int F
   \leq
   \left(e^{-\beta k/2}+(3/4)^k\right)\int G.
\]
For all sufficiently large \(k\), this is at most
\(\gamma_{3.1}^k \int G\) for some \(\gamma_{3.1}=\gamma_{3.1}(\rho)<1\).
This proves \eqref{eq:integral-comparison-noniid}, and hence
\eqref{eq:bbrac-noniid}, in the rational case.

\medskip\noindent
It remains to remove the rationality assumption.  For slices of the discrete
cube, this comes down to the following simple observation.

\begin{lemma}[Cube-slice rationalization]\label{lem:cube-slice-noniid}
Let \(V\subseteq\R^n\) be a linear subspace and let \(a\in\R^n\).  Put
\[
   S=(a+V)\cap\{0,1\}^n.
\]
If \(S\neq\varnothing\), choose \(b_0\in S\) and define
\[
   U=\Span(S-b_0).
\]
Then \(U\) is a rational linear subspace, \(U\subseteq V\), and
\[
   (a+V)\cap\{0,1\}^n=(b_0+U)\cap\{0,1\}^n.
\]
\end{lemma}

\begin{proof}
The space \(U\) is rational because it is spanned by vectors in
\(\{-1,0,1\}^n\), and \(U\subseteq V\) because \(S-b_0\subseteq V\).  The
inclusion
\[
   (a+V)\cap\{0,1\}^n\subseteq (b_0+U)\cap\{0,1\}^n
\]
holds by construction.  Conversely, if \(x\in b_0+U\), then
\(x-b_0\in U\subseteq V\), while \(b_0\in a+V\).  Hence \(x\in a+V\).
\end{proof}

\medskip\noindent
Let now \(V\subseteq\R^n\) be arbitrary with \(\codim V=k\ge k_{3.1}\), and
fix \(a\in\R^n\).  Put
\[
   S=(a+V)\cap\{0,1\}^n.
\]
If \(S=\varnothing\), then \(\PP\{B\in a+V\}=0\).  Otherwise choose
\(b_0\in S\), set
\[
   U=\Span(S-b_0),
\]
and apply Lemma~\ref{lem:cube-slice-noniid}.  Since \(U\subseteq V\),
\(\ell:=\codim U\ge k\).  The rational case gives
\[
   \PP\{B\in a+V\}
   =
   \PP\{B\in b_0+U\}
   \leq
   \gamma_{3.1}^{\ell}\PP\{Z\in U\}
   \leq
   \gamma_{3.1}^{k}\PP\{Z\in V\}.
\]
This completes the proof.
\end{proof}

\section{Comparison for inhomogeneous entries and the thin/thick argument}\label{sec:comparison-thin}

\medskip\noindent
Fix \(b\in(0,1)\), and let
\[
   \rho_0=\frac{1-b}{2}.
\]
Let \(\alpha_{3.1},\gamma_{3.1},k_{3.1}\) be the constants supplied by
Proposition~\ref{prop:bbrac} for \(\rho=\rho_0\), and put
\[
   \lambda_{4.1}=1-2\alpha_{3.1}.
\]

\medskip\noindent
For each pair \((i,j)\), apply Lemma~\ref{lem:atom-decomp} to the law of
\(A_{ij}\).  Thus there exist
\[
   p_{ij}\in[\rho_0,1-\rho_0],
\]
measurable functions \(f_{ij},\delta_{ij}:(0,1)\to\R\) with
\(\delta_{ij}(t)\neq0\), and independent random variables
\[
   T_{ij}\sim {\rm Unif}(0,1),
   \qquad
   \varepsilon_{ij}\sim {\rm Bernoulli}(p_{ij}),
\]
such that
\[
   A_{ij}\stackrel d=f_{ij}(T_{ij})+\delta_{ij}(T_{ij})\varepsilon_{ij}.
\]

\medskip\noindent
Let \(X_j\in\R^n\) denote the \(j\)th column of \(A\), and realize it as
\[
   X_j=f_j+D_j\varepsilon_j,
\]
where
\[
   f_j=(f_{1j}(T_{1j}),\ldots,f_{nj}(T_{nj}))^\top,
   \qquad
   D_j=\operatorname{diag}(\delta_{1j}(T_{1j}),\ldots,\delta_{nj}(T_{nj})),
\]
and
\[
   \varepsilon_j=(\varepsilon_{1j},\ldots,\varepsilon_{nj})^\top.
\]
The columns \(X_1,\ldots,X_n\) are independent, though not identically
distributed.

\medskip\noindent
Let \(Z_j=(Z_{1j},\ldots,Z_{nj})^\top\) have iid coordinates with the lazy
law \eqref{eq:lazy-def-noniid}, independent of everything else, and define
\[
   Y_j=D_j Z_j.
\]

\medskip\noindent
After conditioning on the variables in the decomposition of the
entries, each column has the Bernoulli form treated in
Proposition~\ref{prop:bbrac}.  Averaging that conditional estimate gives the
following comparison for the actual columns of \(A\).

\begin{proposition}[Columnwise relative anticoncentration]\label{prop:RA-column}
For every \(j\in[n]\) and every linear subspace \(V\subseteq\R^n\) of
codimension \(k\ge k_{3.1}\),
\begin{equation}\label{eq:RA-column}
   \sup_{t\in\R^n}\PP\{X_j+t\in V\}
   \leq
   \gamma_{3.1}^k \PP\{Y_j\in V\}.
\end{equation}
\end{proposition}

\begin{proof}
Condition on \(T_{1j},\ldots,T_{nj}\).  Then \(f_j\) and \(D_j\) are fixed,
and \(D_j\) is invertible.  For fixed \(t\in\R^n\),
\[
   X_j+t\in V
   \quad\Longleftrightarrow\quad
   \varepsilon_j\in D_j^{-1}(V-f_j-t).
\]
Proposition~\ref{prop:bbrac}, applied conditionally to the Bernoulli vector
\(\varepsilon_j\), yields
\[
   \PP\{X_j+t\in V\mid T_{1j},\ldots,T_{nj}\}
   \leq
   \gamma_{3.1}^k \PP\{Z_j\in D_j^{-1}V\mid T_{1j},\ldots,T_{nj}\}.
\]
Since \(Z_j\in D_j^{-1}V\) is equivalent to \(D_jZ_j\in V\), averaging over
the \(T_{ij}\)'s proves \eqref{eq:RA-column}.
\end{proof}

\medskip\noindent
We also need a codimension bound for the comparison columns
themselves.

\begin{lemma}[Odlyzko for the comparison columns]\label{lem:Yj-odlyzko}
For every \(j\in[n]\) and every affine subspace \(W\subseteq\R^n\) of
codimension \(r\),
\begin{equation}\label{eq:Yj-odlyzko}
   \PP\{Y_j\in W\}\leq \lambda_{4.1}^r.
\end{equation}
\end{lemma}

\begin{proof}
Condition on \(T_{1j},\ldots,T_{nj}\).  Then \(Y_j=D_jZ_j\), and each
coordinate has largest atom \(1-2\alpha_{3.1}=\lambda_{4.1}\), at zero.  Apply
Lemma~\ref{lem:odlyzko} conditionally and average over the \(T_{ij}\)'s.
\end{proof}

\medskip\noindent
Fix \(1\leq k\leq n\), and put
\[
   m=\left\lceil \frac n2\right\rceil,
   \qquad
   r_\ast=\min\{m,n-k\}.
\]
For \(V\in\mathcal S(n,n-k)\), define
\[
   \rho_j(V)=\PP\{X_j\in V\},
   \qquad
   \sigma_j(V)=\PP\{Y_j\in V\},
\]
\[
   \bar\rho(V)=\frac1n\sum_{j=1}^n \rho_j(V),
   \qquad
   \bar\sigma(V)=\frac1n\sum_{j=1}^n \sigma_j(V),
\]
and, for \(0\leq s\leq n\),
\[
   e_s(V)=\sum_{\substack{S\subseteq[n]\\ |S|=s}} \prod_{j\in S}\rho_j(V).
\]
By Lemma~\ref{lem:grassmann-meas}, each of these maps is Borel on
\(\mathcal S(n,n-k)\).

\medskip\noindent
Set
\[
   \tau=\lambda_{4.1}^{1/4}.
\]
We call \(V\in\mathcal S(n,n-k)\) thin if
\begin{equation}\label{eq:thin-def-noniid}
   e_k(V)\leq \binom{n}{k}\tau^{nk},
\end{equation}
and thick otherwise.

\medskip\noindent
We now split the subspaces into thin and thick classes.  In the thin case,
once a spanning set of columns is fixed, the remaining columns have too
little total mass to land in the same subspace with appreciable probability.

\begin{claim}[Thin subspaces]\label{claim:thin-noniid}
For every \(1\leq k\leq n\),
\begin{equation}\label{eq:thin-bound-noniid}
   \PP\{\rank(X_1,\ldots,X_n)=n-k\text{ and }\Span(X_1,\ldots,X_n)\text{ is thin}\}
   \leq
   \binom{n}{k}^2 \tau^{nk}.
\end{equation}
\end{claim}

\begin{proof}
If \(\Span(X_1,\ldots,X_n)=V\in\mathcal S(n,n-k)\), then some subset
\(I\subseteq[n]\) of size \(n-k\) spans \(V\), and all remaining columns lie
in \(V\).  Fix such an \(I\).  Conditional on \((X_i)_{i\in I}\), the
remaining columns are independent.  Therefore, on the event that
\(V=\Span((X_i)_{i\in I})\) is thin,
\[
   \PP\{X_j\in V\text{ for all }j\notin I\mid (X_i)_{i\in I}\}
   =
   \prod_{j\notin I}\rho_j(V).
\]
This product is one term in \(e_k(V)\), so by
\eqref{eq:thin-def-noniid} it is at most
\(\binom{n}{k}\tau^{nk}\).  There are \(\binom{n}{k}\) possible choices of
\(I\), and summing over them gives \eqref{eq:thin-bound-noniid}.
\end{proof}

\medskip\noindent
For the remaining subspaces, this simple argument is no longer enough.  Here
we compare the original columns to a mixture of the \(Y_j\)'s and then use
Lemma~\ref{lem:weighted-cancel} to remove the weight that appears.  Let
\(\mathcal T_k\subseteq\mathcal S(n,n-k)\) denote the Borel set of thick
subspaces.

\begin{claim}[Thick subspaces]\label{claim:thick-noniid}
Assume \(k\ge k_{3.1}\).  Then
\begin{equation}\label{eq:thick-bound-noniid}
   \PP\{\rank(X_1,\ldots,X_n)=n-k\text{ and }\Span(X_1,\ldots,X_n)\text{ is thick}\}
   \leq
   n2^{3n}\gamma_{3.1}^{nk/2}.
\end{equation}
\end{claim}

\begin{proof}
\medskip\noindent
We use Maclaurin's inequality for elementary symmetric means; see, for
example, Hardy--Littlewood--P\'olya \cite[Chapter~II]{HLP1952}:
\[
   e_s(a_1,\ldots,a_n)\leq \binom{n}{s}
   \left(\frac{a_1+\cdots+a_n}{n}\right)^s
\]
for nonnegative numbers \(a_1,\ldots,a_n\).  Applying this to
\(\rho_1(V),\ldots,\rho_n(V)\), if \(V\in\mathcal T_k\), then
\[
   e_k(V)\leq \binom{n}{k}\bar\rho(V)^k.
\]
Since \(V\) is thick,
\[
   \binom{n}{k}\bar\rho(V)^k
   \ge
   e_k(V)
   >
   \binom{n}{k}\tau^{nk},
\]
and therefore
\begin{equation}\label{eq:thick-implies-bar-rho}
   \bar\rho(V)>\tau^n=\lambda_{4.1}^{n/4}.
\end{equation}

\medskip\noindent
Let \(\widetilde Y\) be the mixture comparison vector obtained by choosing
\(J\) uniformly from \([n]\) and then sampling an independent copy of
\(Y_J\).  Then
\[
   \PP\{\widetilde Y\in V\}=\bar\sigma(V).
\]
Averaging Proposition~\ref{prop:RA-column} over \(j\) gives
\begin{equation}\label{eq:bar-rho-bar-sigma}
   \bar\rho(V)\leq \gamma_{3.1}^k \bar\sigma(V)
\end{equation}
for every \(V\in\mathcal S(n,n-k)\).  Averaging
\eqref{eq:Yj-odlyzko} over \(j\) also gives
\begin{equation}\label{eq:mixture-odlyzko}
   \PP\{\widetilde Y\in W\}\leq \lambda_{4.1}^{\codim W}
\end{equation}
for every affine subspace \(W\subseteq\R^n\).

\medskip\noindent
Let \(\widetilde Y_1,\ldots,\widetilde Y_m\) be iid copies of
\(\widetilde Y\), independent of \(X_1,\ldots,X_n\).  For a Borel set
\(\mathcal A\subseteq\mathcal T_k\), define
\[
   \nu_0(\mathcal A)=
   \PP\{\rank(X_1,\ldots,X_n)=n-k,\ \Span(X_1,\ldots,X_n)\in\mathcal A\}.
\]
Also set
\[
   w(V)=\bar\rho(V)^m,
   \qquad
   V\in\mathcal T_k.
\]
By \eqref{eq:thick-implies-bar-rho}, \(0<w(V)\leq1\) on \(\mathcal T_k\).

\medskip\noindent
For every Borel \(\mathcal A\subseteq\mathcal T_k\),
\begin{align}
   &\PP\{\rank(X_1,\ldots,X_n)=n-k,\ \Span(X_1,\ldots,X_n)\in\mathcal A,
        \ \widetilde Y_1,\ldots,\widetilde Y_m\in \Span(X_1,\ldots,X_n)\}
   \notag\\
   &\qquad
   =
   \int_{\mathcal A}\bar\sigma(V)^m\,d\nu_0(V)
   \ge
   \gamma_{3.1}^{-km}\int_{\mathcal A}w(V)\,d\nu_0(V),
   \label{eq:thick-lower-noniid}
\end{align}
where the inequality follows from \eqref{eq:bar-rho-bar-sigma}.

\medskip\noindent
Let \(F(\mathcal A)\) be the event that
\[
   \Span(X_1,\ldots,X_n,\widetilde Y_1,\ldots,\widetilde Y_m)\in\mathcal A
\]
and this span has dimension \(n-k\).  On \(F(\mathcal A)\), let
\(T\subseteq[m]\) be such that \((\widetilde Y_t)_{t\in T}\) is a basis of
\(\Span(\widetilde Y_1,\ldots,\widetilde Y_m)\), and then let
\(I\subseteq[n]\) be such that \((\widetilde Y_t)_{t\in T}\) together with
\((X_i)_{i\in I}\) is a basis of the full span.  If \(|T|=r\), then
\(|I|=n-k-r\), every unused comparison column lies in
\(\Span((\widetilde Y_t)_{t\in T})\), and every unused original column lies in
the final span.  Thus every outcome in \(F(\mathcal A)\) belongs to at least
one of the events indexed by the triples \((r,T,I)\) in the union bound below.

\medskip\noindent
For fixed \(r\in\{0,\ldots,r_\ast\}\), \(T\subseteq[m]\) with \(|T|=r\), and
\(I\subseteq[n]\) with \(|I|=n-k-r\), define
\[
   W_T=\Span(\widetilde Y_t:t\in T),
   \qquad
   V_{I,T}=\Span\bigl(W_T,(X_i)_{i\in I}\bigr),
\]
and let
\[
   \mu_{I,T}(\mathcal A)=
   \PP\{\dim V_{I,T}=n-k,\ V_{I,T}\in\mathcal A\}.
\]
By Lemma~\ref{lem:span-map}, this is well defined, since on the event
\(\dim V_{I,T}=n-k\) the random span \(V_{I,T}\) belongs to
\(\mathcal S(n,n-k)\).
Then, by the union bound,
\begin{align}
   \PP(F(\mathcal A))
   &\leq
   \sum_{r=0}^{r_\ast}
   \sum_{\substack{T\subseteq[m]\\ |T|=r}}
   \sum_{\substack{I\subseteq[n]\\ |I|=n-k-r}}
   \EE\Bigl[
      \1_{\{V_{I,T}\in\mathcal A,\ \dim V_{I,T}=n-k\}}
      \notag\\
   &\hspace{2.4in}\times
      \PP\{\widetilde Y\in W_T\}^{m-r}
      \prod_{j\notin I}\rho_j(V_{I,T})
   \Bigr].
   \label{eq:upper-pre-odlyzko-noniid}
\end{align}

\medskip\noindent
Now fix \(V\in\mathcal T_k\).  The factor
\(\prod_{j\notin I}\rho_j(V)\) is one term in \(e_{k+r}(V)\).  Hence
the same Maclaurin inequality gives
\begin{equation}\label{eq:ekr-bound-noniid}
   \prod_{j\notin I}\rho_j(V)
   \leq
   e_{k+r}(V)
   \leq
   \binom{n}{k+r}\bar\rho(V)^{k+r}.
\end{equation}
Also, if \(r\leq m-1\), then \(n-r\ge n/2\), so by
\eqref{eq:thick-implies-bar-rho},
\[
   \lambda_{4.1}^{n-r}\leq \lambda_{4.1}^{n/2}\leq \bar\rho(V).
\]
If \(r=m\), then \(m-r=0\), and there is no unused comparison column.  Thus in
all cases,
\[
   \PP\{\widetilde Y\in W_T\}^{m-r}
   \leq
   \bar\rho(V)^{m-r},
\]
using \eqref{eq:mixture-odlyzko}.  Combining this with
\eqref{eq:ekr-bound-noniid},
\[
   \PP\{\widetilde Y\in W_T\}^{m-r}
   \prod_{j\notin I}\rho_j(V)
   \leq
   \binom{n}{k+r}\bar\rho(V)^{m+k}
   \leq
   2^n\bar\rho(V)^m,
\]
since \(\bar\rho(V)\leq1\).

\medskip\noindent
Returning to \eqref{eq:upper-pre-odlyzko-noniid}, we obtain
\[
   \PP(F(\mathcal A))
   \leq
   2^n\sum_{r=0}^{r_\ast}
   \sum_{\substack{T\subseteq[m]\\ |T|=r}}
   \sum_{\substack{I\subseteq[n]\\ |I|=n-k-r}}
   \int_{\mathcal A} w(V)\,d\mu_{I,T}(V).
\]
Define
\[
   \mu_r=\sum_{\substack{T\subseteq[m]\\ |T|=r}}
          \sum_{\substack{I\subseteq[n]\\ |I|=n-k-r}}
          \mu_{I,T}.
\]
Then
\begin{equation}\label{eq:thick-upper-noniid}
   \PP(F(\mathcal A))
   \leq
   2^n \sum_{r=0}^{r_\ast}\int_{\mathcal A} w(V)\,d\mu_r(V).
\end{equation}

\medskip\noindent
The event in \eqref{eq:thick-lower-noniid} is contained in \(F(\mathcal A)\).
Combining \eqref{eq:thick-lower-noniid} and
\eqref{eq:thick-upper-noniid}, we get
\[
   \int_{\mathcal A} w(V)\,d\nu_0(V)
   \leq
   2^n \gamma_{3.1}^{km}
   \sum_{r=0}^{r_\ast}\int_{\mathcal A} w(V)\,d\mu_r(V).
\]
\medskip\noindent
Applying Lemma~\ref{lem:weighted-cancel} on the measurable space
\(\mathcal T_k\), with \(\nu=\nu_0\), \(\sigma=\sum_{r=0}^{r_\ast}\mu_r\), and
the weight \(w(V)=\bar\rho(V)^m\), yields
\[
   \nu_0(\mathcal A)
   \leq
   2^n \gamma_{3.1}^{km}\sum_{r=0}^{r_\ast}\mu_r(\mathcal A)
\]
for every Borel \(\mathcal A\subseteq\mathcal T_k\).

\medskip\noindent
Taking \(\mathcal A=\mathcal T_k\) and using \(\mu_{I,T}(\mathcal T_k)\leq1\),
we obtain
\begin{align*}
   \nu_0(\mathcal T_k)
   &\leq
   2^n\gamma_{3.1}^{km}
   \sum_{r=0}^{r_\ast}
   \binom{m}{r}\binom{n}{n-k-r}\\
   &\leq
   n2^{3n}\gamma_{3.1}^{km}
   \leq
   n2^{3n}\gamma_{3.1}^{nk/2},
\end{align*}
which is \eqref{eq:thick-bound-noniid}.
\end{proof}

\section{Proof of the main theorem}\label{sec:proof-main}

\medskip\noindent
The previous sections already give the desired estimate once \(k\) is large
enough.  The next proposition shows that an exponential singularity bound
handles the remaining finitely many smaller values of \(k\).

\begin{proposition}[Transfer from singularity to corank]\label{prop:transfer}
Fix \(b\in(0,1)\).  For each \(n\ge1\), let \(A=(A_{ij})\) be an
\(n\times n\) random matrix with independent real entries.  Assume that
\begin{align}
   \sup_{x\in\R}\PP\{A_{ij}=x\} &\leq b<1
   \qquad (1\leq i,j\leq n), \label{eq:abstract-atom}\\
   \PP\{\rank A<n\} &\leq C_0\exp(-c_0 n)
   \qquad\text{for all }n\ge n_0 \label{eq:singularity-input-noniid}
\end{align}
for some constants \(C_0>0\), \(c_0>0\), and \(n_0\ge1\).  Then there exists
\(c_{5.1}=c_{5.1}(b,C_0,c_0,n_0)>0\) such that for all \(n\ge1\) and all
\(1\leq k\leq n\),
\begin{equation}\label{eq:all-k-bound-noniid}
   \PP\{\rank A\leq n-k\}\leq \exp(-c_{5.1}nk).
\end{equation}
\end{proposition}

\begin{proof}
Fix \(n\ge1\), and let \(A=(A_{ij})_{1\leq i,j\leq n}\) be the corresponding
matrix from the hypotheses.  Apply the results of
Sections~\ref{sec:bbrac} and \ref{sec:comparison-thin} to \(A\).  Combining
Claims~\ref{claim:thin-noniid} and~\ref{claim:thick-noniid}, for all
\(k\ge k_{3.1}\),
\begin{equation}\label{eq:exact-rank-bound-pre-noniid}
   \PP\{\rank A=n-k\}
   \leq
   \binom{n}{k}^2 \tau^{nk}
   +
   n2^{3n}\gamma_{3.1}^{nk/2}.
\end{equation}
Since \(\binom{n}{k}\leq 2^n\), \(\tau<1\), and \(\gamma_{3.1}<1\), there
exist \(K_{5.1}=K_{5.1}(b)\ge k_{3.1}\) and \(c'_{5.1}=c'_{5.1}(b)>0\) such
that for all \(k\ge K_{5.1}\) and all \(n\ge2\),
\[
   \binom{n}{k}^2 \tau^{nk}+n2^{3n}\gamma_{3.1}^{nk/2}
   \leq
   \exp(-c'_{5.1}nk).
\]
Summing \eqref{eq:exact-rank-bound-pre-noniid} over coranks
\(k,k+1,\ldots,n\) and decreasing \(c'_{5.1}\) if necessary gives
\begin{equation}\label{eq:large-k-final-noniid}
   \PP\{\rank A\leq n-k\}\leq \exp(-c'_{5.1}nk)
\end{equation}
for all \(n\ge2\) and all \(k\ge K_{5.1}\).

\medskip\noindent
We now treat the bounded range \(1\leq k<K_{5.1}\).  Set
\[
   N_{5.1}
   :=
   \max\!\left\{
      n_0,
      \left\lceil \frac{2\log \max\{C_0,1\}}{c_0}\right\rceil
   \right\}.
\]
Then \eqref{eq:singularity-input-noniid} implies that for every
\(n\ge N_{5.1}\),
\[
   \PP\{\rank A<n\}\leq \exp(-c_0 n/2).
\]
Hence for every \(n\ge N_{5.1}\) and every \(1\leq k<K_{5.1}\),
\[
   \PP\{\rank A\leq n-k\}
   \leq
   \exp(-c_0 n/2)
   \leq
   \exp\!\left(-\frac{c_0}{2K_{5.1}}nk\right).
\]

\medskip\noindent
It remains to treat the finitely many pairs \((n,k)\) with \(n<N_{5.1}\) and
\(1\leq k<K_{5.1}\).  Since \(Q(A_{ij})\leq b\), every atom of every
entry has mass at most \(b\).  Applying Lemma~\ref{lem:odlyzko} row by row
gives
\[
   \PP\{\rank A=n\}\ge \prod_{s=1}^n (1-b^s)\ge \prod_{s=1}^\infty (1-b^s)>0.
\]
Indeed, after conditioning on the first \(j-1\) rows, their span has
dimension at most \(j-1\), so the probability that the \(j\)th row lies in
that span is at most \(b^{n-j+1}\).  Therefore
\[
   q_{5.1}(b):=1-\prod_{s=1}^\infty (1-b^s)<1
\]
and
\[
   \PP\{\rank A<n\}\le q_{5.1}(b)
\]
for every matrix satisfying \eqref{eq:abstract-atom}.  Since
\(n<N_{5.1}\) and \(k<K_{5.1}\) in the remaining range, after decreasing
\[
   c_{5.1}=c_{5.1}(b,C_0,c_0,n_0)>0
\]
if necessary we obtain
\[
   q_{5.1}(b)\leq \exp(-c_{5.1}nk)
\]
for all such pairs \((n,k)\).

\medskip\noindent
Finally, decrease \(c_{5.1}\) so that
\[
   c_{5.1}\leq c'_{5.1}
   \qquad\text{and}\qquad
   c_{5.1}\leq \frac{c_0}{2K_{5.1}}.
\]
Then the large-\(k\) bound \eqref{eq:large-k-final-noniid}, the singularity
estimate for \(n\ge N_{5.1}\) and \(1\leq k<K_{5.1}\), and the finite-\(n\)
bound above combine to prove \eqref{eq:all-k-bound-noniid}.
\end{proof}

\medskip\noindent
Theorem~\ref{thm:main-atom} now follows by combining
Proposition~\ref{prop:transfer} with the determinant anticoncentration
statement recalled in the introduction.

\begin{proof}[Proof of Theorem~\ref{thm:main-atom}]
For every \(i,j\) and every \(x\in\R\),
\[
   \PP\{A_{ij}=x\}\leq b.
\]
By the determinant analogue of \cite[Theorem~1.2]{HKS2025perm}, explicitly
noted in the abstract and immediately after Theorem~1.1 there, there is
\(c'=c'(b)>0\) such that
\[
   \PP\{\rank A<n\}\leq \exp(-c'n)
   \qquad\text{for all }n\ge1.
\]
Apply Proposition~\ref{prop:transfer} with \(C_0=1\),
\(c_0=c'\), and \(n_0=1\).  The resulting constant is
\(c_{1.1}=c_{1.1}(b)>0\), and the conclusion is exactly
\eqref{eq:atom-large-corank}.
\end{proof}

\begin{remark}
It remains of interest to extend these results to symmetric random matrices
under comparable hypotheses, in the spirit of recent work of Han
\cite{Han2025} on large-deviation bounds for symmetric subgaussian
matrices.
\end{remark}

\bibliographystyle{alpha}
\bibliography{bibliography}

\end{document}